\documentclass[12pt]{article}
\usepackage[a4paper, total={6in, 8in}, margin={1in}]{geometry}
\usepackage[utf8]{inputenc}

\usepackage{amsmath}
\usepackage{amsthm}
\usepackage{amssymb}
\usepackage{amsfonts}
\usepackage[hidelinks]{hyperref}
\usepackage{scrextend}

\theoremstyle{plain}
\newtheorem{theorem}{Theorem}

\newtheorem{lemma}[theorem]{Lemma}

\theoremstyle{definition}

\theoremstyle{remark}
\newtheorem{remark}[theorem]{Remark}

\title{Quadratic Reciprocity via Vandermonde Matrix}
\author{Hongshen Chua}
\date{}

\begin{document}
\maketitle

\section{Introduction.}

Let $ p $ be an odd prime, and $ a $ an integer not divisible by $ p $. We define the \textit{Legendre symbol} $ \big( \frac{a}{p} \big) $ by setting $ \big( \frac{a}{p} \big) = 1 $ if $ a $ is a square modulo $ p $ and $ \big( \frac{a}{p} \big) = -1 $ otherwise. Euler gave the following simple criterion to determine the value of $ \big( \frac{a}{p} \big) $:	
\begin{align*}
	\bigg( \frac{a}{p} \bigg) \equiv a^{(p - 1)/2} \pmod{p}.
\end{align*}
The goal of this paper is to establish Gauss' famous \textit{law of quadratic reciprocity}, stated in Theorem \ref{Th: Law of quadratic reciprocity}.
\begin{theorem} [Law of quadratic reciprocity] \label{Th: Law of quadratic reciprocity}
	Let $ p $ and $ q $ be distinct odd primes. Then
	\begin{align*}
		\bigg( \frac{q}{p} \bigg) = (-1)^{(p - 1)(q - 1)/4} \bigg( \frac{p}{q} \bigg).
	\end{align*}
\end{theorem}
This result is probably the most proven theorem in mathematics, after Pythagoras' theorem. As of the time of this writing, no fewer than $ 300 $ proofs have been discovered, and they are well documented in Oswald Baumgart's \cite{Baumgart} book.

Our approach combines two methods: a matrix method by Jean Delsarte \cite{Delsarte} and a permutation method by Yegor Ivanovich Zolotarev \cite{Zolotarev}. We are aware of several other papers that have attempted similar approaches, but these papers do not explicitly draw the link between the two aforementioned methods. For a deeper discussion, we refer the reader to $ \S $\ref{Section: Similar proofs} below. 

\section{Vandermonde matrix.} \label{Section: Vandermonde matrix}

Consider the following Vandermonde matrix:
\begin{align*}
	V = V(\zeta^1, \zeta^2, \zeta^3, \ldots, \zeta^p) = 
	\begin{pmatrix}
		1 & \zeta^1 & (\zeta^1)^2 & \cdots & (\zeta^1)^{p - 1} \\
		1 & \zeta^2 & (\zeta^2)^2 & \cdots & (\zeta^2)^{p - 1} \\
		1 & \zeta^3 & (\zeta^3)^2 & \cdots & (\zeta^3)^{p - 1} \\
		\vdots & \vdots & \vdots & \ddots & \vdots \\
		1 & \zeta^p & (\zeta^p)^2 & \cdots & (\zeta^p)^{p - 1}
	\end{pmatrix},
\end{align*}
where $ \zeta^k = e^{2\pi ik/p} $, $ 1 \leq k \leq p $ are the $ p $-th roots of unity. We prove our first lemma. 

\begin{lemma} \label{Lemma: Determinant} 
	$ \det (V)^2 = (-1)^{(p - 1)/2} p^p $.
\end{lemma}

\begin{proof}
	One can compute this by finding $ V^TV $ directly. Indeed, the $ (i, j) $-th entry of $ V^TV $ is 
	\begin{align*}
		\sum_{k = 0}^{p - 1} \zeta^{(i + j - 2)k} = \begin{cases} p, &\text{ if } i + j \equiv 2 \pmod{p}, \\ 0, &\text{ if } i + j \not \equiv 2 \pmod{p}. \end{cases}
	\end{align*}
	Thus 
	\begin{align*}
		\det(V)^2 = \det(V^TV) = 
		\begin{vmatrix}
			p & 0 & 0 & \cdots & 0 & 0 \\
			0 & 0 & 0 & \cdots & 0 & p \\
			0 & 0 & 0 & \cdots & p & 0 \\
			\vdots & \vdots & \vdots & \ddots & \vdots & \vdots \\
			0 & p & 0 & \cdots & 0 & 0
		\end{vmatrix}
		= (-1)^{(p - 1)/2} p^p.
	\end{align*}
\end{proof}

\begin{remark}
	The formula for $ \det(V) $ given in $ \S $\ref{Section: Zolotarev's lemma} shows that we have, in fact, computed the discriminant of $ X^p - 1 $ in Lemma \ref{Lemma: Determinant}. 
\end{remark}

\section{Zolotarev's lemma.} \label{Section: Zolotarev's lemma}

On the other hand, being a Vandermonde matrix, the determinant of $ V $ has the simple form of
\begin{align*}
	\det (V) = \prod_{i < j} (\zeta^j - \zeta^i).
\end{align*}

Note that, under the $ q $-th power map $ \sigma: \zeta \mapsto \zeta^q $, the determinant $ \det (V) $ becomes
\begin{align*}
	\det (\sigma (V)) = \prod_{i < j} (\zeta^{qj} - \zeta^{qi}) = (-1)^\mu \det (V),
\end{align*}
where $ \mu $ is the number of times $ qi \pmod{p} > qj \pmod{p} $, when taken as integers between $ 0 $ and $ p - 1 $. This is exactly the number of inversions under the multiplication map $ \tau_q : x \mapsto qx $ of residue classes modulo $ p $. Therefore, we have 
\begin{align} \label{Eq: 1}
	\det (\sigma (V)) = \text{sgn} (\tau_q) \det (V), 
\end{align}
where $ \text{sgn} (\cdot) $ is the signature of the permutation. Conveniently, this permutation is related to the Legendre symbol through Zolotarev's lemma, given below. 
\begin{lemma} [Zolotarev's lemma] \label{Lemma: Zolotarev}
	Let $ p $ be an odd prime. If $ a \in (\mathbb{Z}/p\mathbb{Z})^* $, then
	\begin{align*}
		\bigg( \frac{a}{p} \bigg) = \text{sgn} (\tau_a),
	\end{align*}
	where $ \tau_a : x \mapsto ax $ is the multiplication-by-$ a $ map on $ (\mathbb{Z}/p\mathbb{Z})^* $.
\end{lemma}

\begin{proof}
	Let $ g $ be a primitive root modulo $ p $. The map $ \tau_g $ is a cycle of length $ p - 1 $, which is even. Hence, $ \tau_g $ is an odd permutation, and $ \text{sgn} (\tau_g) = -1 $. Moreover, since $ g $ is not a square modulo $ p $, we have $ \big( \frac{g}{p} \big) = -1 $. 
	
	For $ a = g^i $, the multiplicativity of both the Legendre symbol and the signature of a permutation implies that
	\begin{align*}
		\bigg( \frac{a}{p} \bigg) = \bigg( \frac{g^i}{p} \bigg) = \bigg( \frac{g}{p} \bigg)^i = \text{sgn} (\tau_g)^i = \text{sgn} (\tau_a).
	\end{align*}
\end{proof}

\section{Law of quadratic reciprocity.} \label{Section: Law of quadratic reciprocity}

Finally, we are ready to prove the law of quadratic reciprocity. 
\begin{proof} [Proof of Theorem \ref{Th: Law of quadratic reciprocity}]
	Observe that $ \det (V), \det (\sigma(V)) \in \mathbb{Z}[\zeta] $. Working modulo $ q\mathbb{Z}[\zeta] $ in the ring $ \mathbb{Z}[\zeta] $, the map $ \sigma $ preserves addition and multiplication. Thus, $ \det(\sigma (V)) \equiv \det(V)^q \pmod{q\mathbb{Z}[\zeta]} $. By Equation \eqref{Eq: 1} and Lemma \ref{Lemma: Zolotarev}, we can write this as 
	\begin{align*}
		\bigg( \frac{q}{p} \bigg) \det (V) \equiv \det (V)^q \pmod{q\mathbb{Z}[\zeta]}.
	\end{align*}
	Multiplying both sides by $ \det (V) $ and using Lemma \ref{Lemma: Determinant} yields 
	\begin{align*}
		\bigg( \frac{q}{p} \bigg) (-1)^{(p - 1)/2} p^p \equiv \det (V)^{q - 1} (-1)^{(p - 1)/2} p^p \pmod{q\mathbb{Z}[\zeta]}.
	\end{align*}
	This simplifies to 
	\begin{align} \label{Eq: 2}
		\bigg( \frac{q}{p} \bigg) \equiv \det (V)^{q - 1} \pmod{q\mathbb{Z}[\zeta]}.
	\end{align}
	On the other hand, since $ \det (V)^{q - 1} = (\det (V)^2)^{(q - 1)/2} $, an application of Lemma \ref{Lemma: Determinant} and Euler's criterion gives
	\begin{align*}
		\det (V)^{q - 1} \equiv (-1)^{(p - 1)(q - 1)/4} p^{p(q - 1)/2} \equiv (-1)^{(p - 1)(q - 1)/4} \bigg( \frac{p^p}{q} \bigg) \pmod{q\mathbb{Z}[\zeta]}.
	\end{align*}
	As $ p $ is odd, this simplifies to 
	\begin{align} \label{Eq: 3}
		\det (V)^{q - 1} \equiv (-1)^{(p - 1)(q - 1)/4} \bigg( \frac{p}{q} \bigg) \pmod{q\mathbb{Z}[\zeta]}.
	\end{align}
	It follows from Equations \eqref{Eq: 2} and \eqref{Eq: 3} that
	\begin{align*}
		\bigg( \frac{q}{p} \bigg) \equiv (-1)^{(p - 1)(q - 1)/4} \bigg( \frac{p}{q} \bigg) \pmod{q\mathbb{Z}[\zeta]}.
	\end{align*}
	Both sides of this congruence are in $ \mathbb{Z} $, and we know that $ \mathbb{Z} \cap q\mathbb{Z}[\zeta] = q\mathbb{Z} $ (see below for details), so the congruence also holds modulo $ q $, i.e., 
	\begin{align*}
		\bigg( \frac{q}{p} \bigg) \equiv (-1)^{(p - 1)(q - 1)/4} \bigg( \frac{p}{q} \bigg) \pmod{q}.
	\end{align*}	
	Since both sides are $ \pm 1 $ and $ q > 2 $, the equivalence is in fact an equality, and we get
	\begin{align*}
		\bigg( \frac{q}{p} \bigg) = (-1)^{(p - 1)(q - 1)/4} \bigg( \frac{p}{q} \bigg),
	\end{align*}
	as desired. 
	
	It remains to justify the fact that $ \mathbb{Z} \cap q\mathbb{Z}[\zeta] = q\mathbb{Z} $. Since $ x^{p - 1} + x^{p - 2} + \ldots + x + 1 $ is irreducible over $ \mathbb{Z} $ and $ \zeta $ is a root of this polynomial, we see that every element in $ \mathbb{Z}[\zeta] $ can be written uniquely as $ a_0 + a_1 \zeta + \ldots + a_{p - 2} \zeta^{p - 2} $, where $ a_0, a_1, \ldots, a_{p - 2} \in \mathbb{Z} $. It follows from this that $ \mathbb{Z} \cap q\mathbb{Z}[\zeta] = q\mathbb{Z} $. 	
\end{proof}

\section{Similar proofs.} \label{Section: Similar proofs}

The matrix $ V $ has been the starting point in several other proofs of the law of quadratic reciprocity, for example, by Jean Delsarte \cite{Delsarte}, Frans Keune \cite{Keune}, and Kaoru Motose \cite{Motose}. Interestingly, Motose uses a modified matrix, which makes his computation somewhat longer. Despite this difference, our method closely resembles Motose's in that it is carried out over the complex numbers. In contrast, Delsarte and Keune work within a finite field. Both approaches are ultimately equivalent. 

A key distinction in our proof is the realization that the change of sign of $ \det(V) $ under $ \sigma $ is already captured by Zolotarev's lemma (Lemma \ref{Lemma: Zolotarev}). The other papers follow the same reasoning as in the proof of Zolotarev's lemma above, but without referencing it. Nevertheless, all the aforementioned papers serve as valuable supplements to ours. 

Additionally, we have conducted a quick survey of proofs based on Zolotarev's lemma. To our knowledge, these proofs rely on a combinatorial argument, none of which involve the matrix $ V $. Thus, we present our proof as a bridge that connects both approaches. 	

\bibliographystyle{vancouver}

\end{document}